\documentclass[11pt]{amsart}
\usepackage{eurosym}
\usepackage{amssymb}
\usepackage{amsmath}
\usepackage{amsfonts}
\usepackage{geometry}

\newtheorem{theorem}{Theorem}
\theoremstyle{plain}

\newtheorem{definition}{Definition}
\newtheorem{example}{Example}

\newtheorem{proposition}{Proposition}
\newtheorem{remark}{Remark}

\numberwithin{equation}{section}
\input{tcilatex}

\def\DD{\mathcal{D}}

\def\bX{{\mathbf X}}
\def\bY{{\mathbf Y}}

\def\XX{{\mathbb X}}
\def\YY{{\mathbb Y}}

\def\/{|\!|\!|}

\begin{document}
\title{Physical Brownian motion in magnetic field as rough path}
\author{Peter Friz, Paul Gassiat, Terry Lyons}

\begin{abstract}
The indefinite integral of the homogenized Ornstein-Uhlenbeck process is a
well-known model for physical Brownian motion, modelling the behaviour of an
object subject to random impulses [L. S. Ornstein, G. E. Uhlenbeck: On the
theory of Brownian Motion. In: Physical Review. 36, 1930, 823-841]. One can
scale these models by changing the mass of the particle and in the small
mass limit one has almost sure uniform convergence in distribution to the
standard idealized model of mathematical Brownian motion. This provides one
well known way of realising the Wiener process. However, this result is less
robust than it would appear and important generic functionals of the
trajectories of the physical Brownian motion do not necessarily converge to
the same functionals of Brownian motion when one takes the small mass limit.
In presence of a magnetic field the area process associated to the physical
process converges - but not to L\'evy's stochastic area. As this area is
felt generically in settings where the particle interacts through force
fields in a nonlinear way, the remark is physically significant and
indicates that classical Brownian motion, with its usual stochastic
calculus, is not an appropriate model for the limiting behaviour.

We compute explicitly the area correction term and establish convergence, in
the small mass limit, of the physical Brownian motion in the rough path
sense. The small mass limit for the motion of a charged particle in the
presence of a magnetic field is, in distribution, an easily calculable, but
"non-canonical" rough path lift of Brownian motion. Viewing the trajectory
of a charged Brownian particle with small mass as a rough path is
informative and allows one to retain information that would be lost if one
only considered it as a classical trajectory. We comment on the importance
of this point of view.

\end{abstract}

\subjclass{Primary 60H99}
\keywords{Physical Brownian motion, Homogenization, L\'{e}vy's area}
\thanks{The research of PF and PG is supported by the European Research
Council under the European Union's Seventh Framework Programme
(FP7/2007-2013) / ERC grant agreement nr. 258237. The research of TL is
supported by EPSRC grant EP/H000100/1 and the European Research Council
under the European Union's Seventh Framework Programme (FP7-IDEAS-ERC) / ERC
grant agreement nr. 291244.}
\maketitle

\section{Introduction}

Newton's second law for a particle in $\mathbb{R}^{3}$ with mass $m$, and
position $x=x\left( t\right) ,$ (for simplicity: constant) frictions $\alpha
_{1},\alpha _{2},\alpha _{3}>0$ in the coordinate axis, subject to a ($3$%
-dimensional) white noise in time, $\xi$, where $\xi=\xi(t)$ is the
distributional derivative of $W$, a (mathematical)\ Brownian motion or
Wiener process, reads%
\begin{equation}
m\ddot{x}=-A\dot{x}+ \xi  \label{newton}
\end{equation}%
where $A=diag\left( \alpha _{1},\alpha _{2},\alpha _{3}\right) $.
Orthonormal change of coordinates implies that the "correct" assumption for $%
A$ is to be symmetric with strictly positive spectrum,%
\begin{equation*}
\sigma \left( A\right) \subset \left( 0,\infty \right) .
\end{equation*}%
The process $x\left( t\right) $ describes what is known as \textit{physical
Brownian motion}. It is well-known that in small mass regime, $m<<1$, of
obvious physcial relevance when dealing with particles, a good approximation
is given by (mathematical) \textit{Brownian motion}; to see this formally,
it suffices to take $m=0$ in (\ref{newton}) in which case $Ax$ is a standard
Brownian motion in $\mathbb{R}^{3}$.

Let us now assume that our particle (with position $x$ and momentum $m\dot{x}
$) carries an electric charge $q\neq 0$ and moves in a magnetic field $%
\mathbb{B}$ which we assume to be constant. Recall that such a particle
experiences a sideways force ("Lorentz force") that is proportional to the
strength of the magnetic field, the component of the velocity that is
perpendicular to the magnetic field and the charge of the particle.%
\begin{equation*}
F_{Lorentz}=q\dot{x}\times \mathbb{B}\text{.}
\end{equation*}%
When $\mathbb{B}$ is constant, which we assume for simplicity, the Lorentz
force experienced by the particle (at time $t$) can be written as as linear
function of velocity $\dot{x}=\dot{x}\left( t\right) $, namely $qB\dot{x}$
for some anti-symmetric matrix $B$. In other words, the dynamics for \textit{%
physical Brownian motion in a magnetic field} are given by 
\begin{eqnarray*}
m\ddot{x} &=&-A\dot{x}+qB\dot{x}+ \xi \\
&\equiv &-M\dot{x}+\xi)
\end{eqnarray*}%
where $M$ $=$ $(A+qB)$ is such that all eigenvalues of $M\in \mathbb{R}%
^{n\times n}$ have strictly positive real part (one may still think $n=3$,
but the subsequent analysis works for any dimension $n$). Note that these
second order dynamics can be rewritten as evolution equation for the \textit{%
momentum} 
\begin{equation*}
p\left( t\right) =m\dot{x}\left( t\right),
\end{equation*}
indeed, 
\begin{equation*}
\dot{p}=-M\dot{x}+\xi = -\frac{1}{m}M{p}+\dot{W},
\end{equation*}%
and we shall take this point of view when rewriting the dynamics in term of
standard stochastic differential equations. As before, when $m=0$, the
process $Mx=W$ is a bona fide (i.e. mathematical) $3$-dimensional Brownian
motion and one may think that little has changed, appart from the covariance
matrix of the resulting Brownian motion. And indeed, writing $\xi =\dot{W}$,
and assuming $x_{0}=W_{0}=0$, it is easy to see that%
\begin{equation*}
W_{t}-Mx_{t}=\int_{0}^{t}\left( -M\dot{x}_{s}+\xi _{s}\right) ds=\int_{0}^{t}%
\dot{p}=(p_{t}-p_{0})\rightarrow 0\text{ as mass }m\rightarrow 0\text{.}
\end{equation*}%
Note that 
\begin{equation*}
p_{0}=m\dot{x}\left( 0\right) \rightarrow 0
\end{equation*}
as $m\rightarrow 0$, whenever initial velocity remains uniformly bounded, so
the statement here is that $p_{t}\rightarrow 0$ as $m\rightarrow 0$ and one
can easily see that this convergence is uniform (we are only interested in a
fixed time horizon, say $\left[ 0,T\right] $). However, the momentum may
have quite non-trivial effects as \textit{control.} Recalling that
controlled system (differential equations, integrals ...) behave in a robust
fashion precisly under rough path metrics, the essence of which has been
briefly summarized for the reader's convenience in the appendix, the
following lemma tells us that momentum does not at all converge to a trivial
limit. The situation is reminiscent a well-known deterministic example,
taken from \cite{L98}, where the path%
\begin{equation*}
Z_{t}:=\frac{1}{\sqrt{m}}\exp \left( 2\pi itm\right) \in \mathbb{C\simeq R}%
^{2},
\end{equation*}%
converges to a non-trivial "pure area" rough path as $m\rightarrow 0$.
(See, however, remark \ref{rmkFBM} where we emphasize the special
role of Brownian motion, notably the need for its intrinisic irregularity,
in the proposition and theorem below.)
\begin{proposition}
One has (the iterated integrals are understood in It\^{o} sense)%
\begin{equation*}
\left( P_{s,t},\mathbb{P}_{s,t}\right) :=\left(
\int_{s}^{t}dp,\int_{s}^{t}\int_{s}^{r}dp\otimes dp\right) \rightarrow
\left( 0,\frac{MC-CM^{\ast }}{2}\left( t-s\right) \right) \text{ as }%
m\rightarrow 0\text{.}
\end{equation*}%
and $C$ is the (symmetric) $n\times n$ matrix defined by 
\begin{equation*}
C=\int_{0}^{\infty }e^{-Ms}e^{-M^{\ast }s}ds.
\end{equation*}
More precisely, the convergence holds in the following strong sense: for any 
$\alpha \in (1/3,1/2)$, 
\begin{equation}
\sup_{s,t\in \lbrack 0,T]}\frac{\left\vert P_{s,t}\right\vert }{%
|t-s|^{\alpha }}+\sup_{s,t\in \lbrack 0,T]}\frac{\left\vert \mathbb{P}_{s,t}-%
\frac{MC-CM^{\ast }}{2}\left( t-s\right) \right\vert }{|t-s|^{2\alpha }}%
\rightarrow 0\mbox{ in }L^{q},\mbox{ as }m\downarrow 0.
\end{equation}%
(This is precisely what is meant by convergence in $\alpha $-H\"{o}lder
rough path metric.)
\end{proposition}

In view of this proposition one suspects (correctly) that the convergence of
physical Brownian motion to Brownian motion is also non-trivial if one
thinks how Brownian particles act as controls, i.e. as rough paths. More
specifically, one expects a limit in which \textit{L\'{e}vy's stochastic
area }is perturbed by a term proportional to 
\begin{equation*}
(MC-CM^{\ast })/2,
\end{equation*}%
the (anti-symmetric) matrix which effectively described the pure area
(rough path) limit of the previous proposition. Let us insist, however, that
such a statement (i.e. the content of the following theorem) is \textit{not}
a corollary of the above since, in general,%
\begin{equation*}
\int \left( Z-Z^{n}\right) d\left( Z-Z^{n}\right) \rightarrow 0\nRightarrow
\int Z^{n}dZ^{n}\rightarrow \int ZdZ.
\end{equation*}
Indeed, if one thinks of $Z^{n},Z$ as rough paths, their (formal) iterated
integral are meaningful by definition of a rough path. In contrast, the
iterated integrals of $Z$ agains $Z^{n}$ will not even be, in general,
well-defined. We are now ready to state our main result.

\begin{theorem}
\label{thm:main}Let $M$ be a square matrix in dimension $n$ such that all
its eigenvalues have strictly positive real part. Let $W$ be a $n$%
-dimensional standard Brownian motion, $m$ ("mass") as strictly positive
scalar and consider the stochastic differential equations 
\begin{eqnarray}
dX &=&\frac{1}{m}Pdt\text{ \ \ (position)} \\
dP &=&-\frac{1}{m}MPdt+dW\text{ (momentum).}
\end{eqnarray}%
with (for simplicity only) zero initial position and momentum. Let $\mathbf{W%
}$ $=$ $(W,\mathbb{W}) $ be the natural rough path lift of $W$, where $%
\mathbb{W}_{s,t}=\int_{s}^{t}\left( W_{r}-W_{s}\right) \otimes \circ dW_{r}$%
. Define also $\widehat{\mathbf{W}}$ $=$ $(W,\widehat{\mathbb{W}})$, where 
\begin{equation}
\widehat{\mathbb{W}}_{s,t}=\mathbb{W}_{s,t}+(t-s)\frac{1}{2}(MC-CM^{\ast }),
\end{equation}%
and $C$ is as in the previous proposition. Then, as $m\rightarrow 0$, $MX$ \
converges to $\widehat{\mathbf{W}}$ in $L^{q}$ and $\alpha $-H\"{o}lder
rough path topology, for any $q$ $\geq $ $1$ and $\alpha $ $\in $ $(1/3,1/2)$%
. More precisely, we have 
\begin{equation}  \label{eqRPConv}
\sup_{s,t\in \lbrack 0,T]}\frac{\left\vert MX_{s,t}-W_{s,t}\right\vert }{%
|t-s|^{\alpha }}+\sup_{s,t\in \lbrack 0,T]}\frac{\left\vert
\int_{s}^{t}MX_{s,r}\otimes d(MX)_{r}-\widehat{\mathbb{W}}_{s,t}\right\vert 
}{|t-s|^{2\alpha }}\rightarrow 0\mbox{ in }L^{q}
\end{equation}%
as $m\rightarrow 0$ and this convergence is of rate arbitrarily close to $%
1/2-\alpha $.
\end{theorem}

\begin{remark}
In view of the tensorial transformation behaviour of iterated integrals, (%
\ref{eqRPConv}) is plainly equivalent to 
\begin{eqnarray*}
X_{s,t} &\rightarrow& M^{-1}W_{s,t}, \\
\int_s^t X_{s,u} \otimes dX_u &\rightarrow& M^{-1}\mathbb{W}_{s,t}
(M^{-1})^\ast + (t-s)\frac{1}{2}(C(M^{-1})^\ast-M^{-1}C);
\end{eqnarray*}
in $\alpha $-H\"{o}lder rough path metric.
\end{remark}

\begin{remark}
One has 
\begin{equation*}
\widehat{\mathbf{W}} = \mathbf{W}
\end{equation*}
if and only if $M$ is symmetric. To see this, note that $MC$ is symmetric
(hence equal to $CM^{\ast }$) if and only if $M$ is symmetric, using
symmetry and invertibility of $C$.
\end{remark}

\begin{remark}
The framework of Gaussian rough paths \cite{FVGauss} and \cite[Ch. 15]%
{FVforth}, which plays keyrole in non-Markovian H\"{o}rmander theory \cite%
{CF, CHLT} and some recent breakthroughs in non-linear SPDE theory \cite%
{H,H2} is only applicable if $M$ is symmetric since then one can diagonalize
the dynamics such as to have Gaussian driving signals with independent
components. In this case, upon checking some uniform variation bounds on the
covariance, it could be used to see the convergence (\ref{eqRPConv}) to $%
\mathbf{W}$, the "natural" rough path lift of Brownian motion; \cite[Ch. 15]%
{FVforth}. But since $\widehat{\mathbf{W}}$ is not natural when $Anti\left(
M\right) \neq 0$, we here also provide an explicit example which illustrates
the necessity of the assumptions put forward in \cite{FVGauss, FVforth}.
\end{remark}

\begin{remark}
If $M$ is normal, i.e. 
\begin{equation*}
[M,M^{\ast }] = 0,
\end{equation*}
then the difference between $\widehat{\mathbb{W}}$ and $\mathbb{W}$ has a
somewhat simpler expression. Indeed , we compute 
\begin{equation*}
C=\frac{1}{2}\func{Sym}(M)^{-1},
\end{equation*}
and since $C$ commutes with $M$, we get 
\begin{equation}
\frac{1}{2}(MC-CM^{\ast })=\frac{1}{2}\func{Anti}(M)\func{Sym}(M)^{-1}.
\label{RmkAreaCorrSimple}
\end{equation}
Similarly, the area correction for $X$ is seen to simpify to 
\begin{equation}
\frac{1}{2}(C(M^{-1})^\ast-M^{-1}C)= -\frac{1}{2}\func{Anti}(M^{-1}) \func{%
Sym}(M)^{-1} .  \label{RmkAreaCorrSimpleX}
\end{equation}
\end{remark}

\begin{remark}
It would be possible to "unit" the above proposition and theorem in saying
that the physcial Brownian motion is a "good" approximation in the sense of 
\cite{CFV} to the limiting rough path $\widehat{\mathbf{W}}$.
\end{remark}

\begin{remark}
Incidentally, the rate "anything less than $1/2-\alpha$" is also the rate of
convergence (in $\alpha $-H\"{o}lder rough path metric) for piecewise linear
"Wong--Zakai" approximations to Brownian motion obtained in \cite{HN} and
optimally of this rate is known. In both cases, these rates are obtained as
a rather mechanical consequence of good moment estimates, cf. \cite[Thm A.13]%
{FVforth}, so that we also believe our rates to be optimal.
\end{remark}

The following example is taken from Pavliotis--Stuart, Hairer \cite[Section
11.7.7]{PS}. We note that the PDE argument (based on multiscale techniques)
offered in \cite[Section 11.7.7]{PS} can only give convergence of the
finite-dimensional distributions, tightness - especially in rough path
topology - will require additional and non-trivial estimates (which are
implied for our work below). And of course, strong convergence, available to
us because of a natural coupling of physical Brownian motion $X$ and $W=\dot{%
\xi}$, is out of reach with PDE methods.

\begin{example}
Take $n=2, \alpha \in \mathbf{R}$ and 
\begin{equation*}
M=Id-\alpha \left( 
\begin{array}{cc}
0 & -1 \\ 
1 & 0%
\end{array}%
\right)
\end{equation*}%
so that 
\begin{equation*}
M^{-1}=\frac{1}{ 1+\alpha ^{2}} \left( 
\begin{array}{cc}
1 & - \alpha \\ 
\alpha & 1%
\end{array}%
\right).
\end{equation*}%
Note $\func{Sym} (M) = Id$. Then, from (\ref{RmkAreaCorrSimpleX}), the
area correction of $X$, in the $m \rightarrow 0$ limit, equals%
\begin{equation}
\frac{-1}{2 (1+\alpha^2)} \left( 
\begin{array}{cc}
0 & - \alpha \\ 
\alpha & 0%
\end{array}%
\right) =\frac{\alpha}{2(1+\alpha^2)}\left( e_{1}\otimes e_{2}-e_{2}\otimes
e_{1}\right)  \label{ExPS}
\end{equation}%
where $e_{1},e_{2}$ denotes the standard basis of $\mathbf{R}^{2}$. This
agrees precisely with \cite[(11.7.28c)]{PS}.
\end{example}

Let us conclude this introduction by noting that the example of physical
Brownian motion under influence of a magnetic field as a rough path has some
history. Indeed, it appeared as motivation (but without much details) in
several presentations, including some by the last-named author in early 2000.


\section{Proof of theorem \protect\ref{thm:main}}

We first give a proof of the rough path convergence without rates, based on
the ergodic theorem. The adaptions which lead to the announced rates (and
also bypass the use of the ergodic theorem!) are then explained in details
in the remark following the proof.

\begin{proof}
In order to exploit Brownian scaling, it is convenient to set 
\begin{equation*}
m=\varepsilon ^{2}
\end{equation*}
and then $Y^{\varepsilon }$ as 
rescaled momentum,%
\begin{equation*}
Y^{\varepsilon }=P/\varepsilon \text{.}
\end{equation*}%
We shall also write $X^{\varepsilon }=X$, to emphasize dependence on $%
\varepsilon $. We then have 
\begin{eqnarray}
dY^{\varepsilon } &=&-\varepsilon ^{-2}MY^{\varepsilon }dt+\varepsilon
^{-1}dW  \label{eqY2} \\
dX^{\varepsilon } &=&\varepsilon ^{-1}Y^{\varepsilon }dt,  \label{eqZ2}
\end{eqnarray}

By assumption, there exists $\lambda >0$ s.t. the real part of every
eigenvalue of $M$ is (strictly) bigger than $\lambda $. For later reference,
we note that this implies the estimate 
\begin{equation*}
|\exp (-\tau M)|=O(\exp (-\lambda \tau ))
\end{equation*}
as $\tau \rightarrow \infty $. For fixed $\varepsilon $, define the Brownian
motion 
\begin{equation*}
\tilde{W}_{\cdot }=\varepsilon W_{\varepsilon ^{-2}\cdot },
\end{equation*}%
and consider the SDEs 
\begin{eqnarray*}
d\tilde{Y} &=&-M\tilde{Y}dt+d\tilde{W} \\
d\tilde{X} &=& \tilde{Y}dt;
\end{eqnarray*}%
note that the law of the solutions does not depend on $\varepsilon $.
Furthermore, when solved with identical initial data as \eqref{eqY2}-%
\eqref{eqZ2}, we have pathwise equality 
\begin{equation}
\left( Y_{\cdot }^{\varepsilon },X_{\cdot }^{\varepsilon }/\varepsilon
\right) =\tilde{Y}_{\varepsilon ^{-2}\cdot },\tilde{X}_{\varepsilon
^{-2}\cdot }
\end{equation}

Thanks to our assumption on $M$, $\tilde{Y}$ is ergodic; the stationary
solution has (zero mean, Gaussian) law 
\begin{equation*}
\nu \sim \mathcal{N}(0,C)
\end{equation*}
for some covariance matrix $C$. To compute it, write down the stationary
solution%
\begin{equation*}
\tilde{Y}_{t}^{\text{stat}}=\int_{-\infty }^{t}e^{-M\left( t-s\right)
}dW_{s};
\end{equation*}%
for each $t$ (and in particular $t=0$) the law of $\tilde{Y}_{t}^{\text{stat}%
}$ is precisely $\nu $. We then see that 
\begin{equation*}
C=\mathbb{E}\left( \tilde{Y}_{0}^{\text{stat}}\otimes \tilde{Y}_{0}^{\text{%
stat}}\right) =\int_{-\infty }^{0}e^{-M\left( -s\right) }e^{-M^{\ast }\left(
-s\right) }ds=\int_{0}^{\infty }e^{-Ms}e^{-M^{\ast }s}ds.
\end{equation*}%
Since 
\begin{equation*}
\sup_{0\leq t<\infty }{\mathbb{E}}|\tilde{Y}_{t}^{2}|<\infty,
\end{equation*}
it is clear that 
\begin{equation*}
\varepsilon \tilde{Y}_{\varepsilon ^{-2}t} = \varepsilon Y_{t}^{\varepsilon
} \rightarrow 0
\end{equation*}
in $L^{2}$ uniformly in $t$ (and hence in $L^{q}$ for any $q<\infty $).
Noting that 
\begin{equation*}
MX_{t}^{\varepsilon }=W_{t}-\varepsilon Y_{0,t}^{\varepsilon },
\end{equation*}
the first part of the proposition is now obvious. Moreover, by the ergodic
theorem\footnote{%
See e.g. \cite{Stroock} (p.421) or \cite{Kal} (p.409)}, 
\begin{equation}
\int_{0}^{t}f(Y_{t}^{\varepsilon })dt\rightarrow t\int f(y)\nu
(dy),\;\;\;\;\;\mbox{ in }L^{q},\mbox{ any }q<\infty ,  \label{eqErgodic}
\end{equation}%
for all reasonable test function $f$. We shall only use it for quadratics%
\footnote{%
The ergodic theorem in the references we have cited only applies to bounded $%
f$, but it is easy to see by a truncation argument that \eqref{eqErgodic}
still holds as long as $(f(Y_{s}))_{s\geq 0}$ is bounded in any $L^{q}$ and $%
\nu $ has finite moments of all order.}.\ \ Using 
\begin{equation*}
dX^{\varepsilon }=\varepsilon ^{-1}Y^{\varepsilon }dt
\end{equation*}
we can then write 
\begin{eqnarray*}
\int_{0}^{t}MX_{s}^{\varepsilon }\otimes d(MX^{\varepsilon })_{s}
&=&\int_{0}^{t}MX_{s}^{\varepsilon }\otimes dW_{s}-\varepsilon
\int_{0}^{t}MX_{s}^{\varepsilon }\otimes dY_{s}^{\varepsilon } \\
&=&\int_{0}^{t}MX_{s}^{\varepsilon }\otimes dW_{s}-MX_{t}^{\varepsilon
}\otimes (\varepsilon Y_{t}^{\varepsilon })+\varepsilon
\int_{0}^{t}d(MX^{\varepsilon })_{s}\otimes Y_{s}^{\varepsilon } \\
&=&\int_{0}^{t}MX_{s}^{\varepsilon }\otimes dW_{s}-MX_{t}^{\varepsilon
}\otimes (\varepsilon Y_{t}^{\varepsilon })+\int_{0}^{t}MY_{s}^{\varepsilon
}\otimes Y_{s}^{\varepsilon }ds \\
&\rightarrow &\int_{0}^{t}W_{s}\otimes dW_{s}-0+t\int My\otimes y\nu (dy) \\
&=&\int_{0}^{t}W_{s}\otimes dW_{s}+tMC \\
&=&\mathbb{W}_{0,t}+t\left( MC-\frac{1}{2}I\right)
\end{eqnarray*}%
where the convergence is in $L^{q}$ for any $q\geq 2$. By considering the
symmetric part of the above equation, 
\begin{equation*}
\frac{1}{2}(MX_{t}^{\varepsilon })\otimes (MX_{t}^{\varepsilon })\rightarrow 
\frac{1}{2}W_{t}\otimes W_{t}+\func{Sym}\left( MC-\frac{1}{2}I\right) ,
\end{equation*}%
we see that 
\begin{equation*}
MC-\frac{1}{2}I
\end{equation*}
has symmetric part $0$, i.e. is antisymmetric, and hence also equals 
\begin{equation*}
\frac{1}{2}\left( MC-CM^{\ast }\right).
\end{equation*}
This settles pointwise convergence, in the sense that 
\begin{equation*}
\left( MX_{t}^{\varepsilon },\int_{0}^{t}MX_{s}^{\varepsilon }\otimes
d(MX^{\varepsilon })_{s}\right) \rightarrow \left( W_{t},\widehat{\mathbb{W}}%
_{0,t}\right) .
\end{equation*}%
In view of \cite{FVforth}, Proposition A.15, the rough path convergence %
\eqref{eqRPConv} will follow once we have checked that for any $q<\infty $, 
\begin{eqnarray*}
\sup_{\varepsilon \in (0,1]}E\left[ \left\vert X_{s,t}^{\varepsilon
}\right\vert ^{q}\right] &\leq &C_{q}\left\vert t-s\right\vert ^{\frac{q}{2}}
\\
\sup_{\varepsilon \in (0,1]}E\left[ \left\vert \int_{s}^{t}X_{s,\cdot
}^{\varepsilon }\otimes dX^{\varepsilon }\right\vert ^{q}\right] &\leq
&C_{q}\left\vert t-s\right\vert ^{q}.
\end{eqnarray*}%
First, since $X$ is Gaussian, it follows from integrability properties of
the first two Wiener-Ito chaos, that it is enough to show it for $q=2$.
Secondly, we note that the desired estimates are a consequence of the
following : 
\begin{eqnarray}
{\mathbb{E}}\left[ \left\vert \tilde{X}_{s,t}\right\vert ^{2}\right] &\leq &(%
\mbox{const.})|t-s|,  \label{eqHol1} \\
E\left[ \left\vert \int_{s}^{t}\tilde{X}_{s,\cdot }\otimes d\tilde{X}%
\right\vert ^{2}\right] &\leq &(\mbox{const.})\left\vert t-s\right\vert ^{2},
\label{eqHol2}
\end{eqnarray}%
where the constants must be uniform over $t,s$ $\in $ $(0,\infty )$. Indeed,
this follows directly from writing 
\begin{equation*}
{\mathbb{E}}\left[ \left\vert X_{s,t}^{\varepsilon }\right\vert ^{2}\right] =%
{\mathbb{E}}\left[ \left\vert \varepsilon \tilde{X}_{\varepsilon
^{-2}s,\varepsilon ^{-2}t}\right\vert ^{2}\right] \leq (\mbox{const.}%
)\varepsilon ^{2}\left\vert \varepsilon ^{-2}t-\varepsilon ^{-2}s\right\vert
=(\mbox{const.})|t-s|,
\end{equation*}%
(note uniformity in $\varepsilon $), and similarly for the second moment of
the iterated integral.

In order to check \eqref{eqHol1}, it is enough to note $M\tilde{X}_{s,t}=%
\tilde{W}_{s,t}-\tilde{Y}_{s,t}$, combined with the estimate 
\begin{eqnarray*}
{\mathbb{E}}\left[ |\tilde{Y}_{s,t}|^{2}\right] &=&{\mathbb{E}}\left[
\left\vert (e^{-M(t-s)}-I)\tilde{Y}_{s}\right\vert ^{2}\right]
+\int_{s}^{t}Tr(e^{-Mu}e^{-M^{\ast }u})du \\
&\leq &C(M)\left\vert t-s\right\vert ,
\end{eqnarray*}%
using the fact that $Re(\sigma (M))\subset (0,+\infty )$. For the second
one, we write 
\begin{eqnarray*}
E\left[ \left\vert \int_{s}^{t}(\tilde{X}_{s,u})^{i}d(\tilde{X}%
_{u})^{j}\right\vert ^{2}\right] &=&{\mathbb{E}}\left[ \left\vert
\int_{s}^{t}\int_{s}^{u}\tilde{Y}_{r}^{i}\tilde{Y}_{u}^{j}dr\,du\right\vert
^{2}\right] \\
&=&\int_{[s,t]^{4}}{\mathbb{E}}\left[ \tilde{Y}_{r}^{i}\tilde{Y}_{u}^{j}%
\tilde{Y}_{q}^{i}\tilde{Y}_{v}^{j}\right] \mathbf{1}_{\{r\leq u;q\leq
v\}}dr\,du\,dq\,dv \\
&\leq &\int_{[s,t]^{4}}\left( \left\vert {\mathbb{E}}\left[ \tilde{Y}_{r}^{i}%
\tilde{Y}_{u}^{j}\right] \right\vert \left\vert {\mathbb{E}}\left[ \tilde{Y}%
_{q}^{i}\tilde{Y}_{v}^{j}\right] \right\vert +\left\vert {\mathbb{E}}\left[ 
\tilde{Y}_{r}^{i}\tilde{Y}_{q}^{i}\right] \right\vert \left\vert {\mathbb{E}}%
\left[ \tilde{Y}_{u}^{j}\tilde{Y}_{v}^{j}\right] \right\vert \right. \\
&&\left. \;\;\;\;\;\;\;\;\;\;\;\;+\left\vert {\mathbb{E}}\left[ \tilde{Y}%
_{r}^{i}\tilde{Y}_{v}^{j}\right] \right\vert \left\vert {\mathbb{E}}\left[ 
\tilde{Y}_{u}^{j}\tilde{Y}_{q}^{i}\right] \right\vert \right) dr\,du\,dq\,dv
\\
&\leq &C_{0}\left( \int_{[s,t]^{2}}\left\vert {\mathbb{E}}\left[ \tilde{Y}%
_{r}\otimes \tilde{Y}_{u}\right] \right\vert dr\,du\right) ^{2} \\
&\leq &C_{1}\left( \int_{[s,t]^{2}}\left\vert {\mathbb{E}}\left[ \tilde{Y}%
_{r}\otimes \tilde{Y}_{u}\right] \right\vert \mathbf{1}_{\{r\leq
u\}}dr\,du\right) ^{2},
\end{eqnarray*}%
where we have used the fact that $\tilde{Y}$ is Gaussian ("Wick's formula").
But for $r\leq u$, ${\mathbb{E}}\left[ \tilde{Y}_{u}\big\vert\tilde{Y}_{r}%
\right] =e^{-M(u-r)}\tilde{Y}_{r}$, so that

\begin{eqnarray*}
&&\left( \int_{[s,t]^{2}}\left\vert {\mathbb{E}}\left[ \tilde{Y}_{r}\otimes 
\tilde{Y}_{u}\right] \right\vert \mathbf{1}_{\{r\leq u\}}dr\,du\right) ^{2}
\\
&=&\left( \int_{[s,t]^{2}}\left\vert {\mathbb{E}}\left[ \tilde{Y}_{r}\otimes
e^{-M(u-r)}\tilde{Y}_{r}\right] \right\vert \mathbf{1}_{\{r\leq
u\}}dr\,du\right) ^{2} \\
&\leq &C_{2}\left( \int_{s}^{t}\left( \int_{r}^{t}e^{-\lambda
(u-r)}du\right) {\mathbb{E}}\left[ \left\vert \tilde{Y}_{r}\right\vert ^{2}%
\right] dr\right) ^{2}\leq C_{3}(t-s)^{2},
\end{eqnarray*}%
recalling that $|\exp (-\tau M)|=O(\exp (-\lambda \tau ))$, and %
\eqref{eqHol2} is proved.
\end{proof}

\begin{remark}
\textbf{(Rates) }The use of the ergodic theorem can be avoided by checking
"directly" that%
\begin{equation*}
\int_{0}^{t}Y_{s}^{\varepsilon }\otimes Y_{s}^{\varepsilon }ds\rightarrow
tC=t\mathbb{E}\left( \tilde{Y}_{0}^{\text{stat}}\otimes \tilde{Y}_{0}^{\text{%
stat}}\right)
\end{equation*}%
in $L^{q}$ $\forall q<\infty $, with a proof similar to the proof of the
inequality \eqref{eqHol2}. Assume for simplicity that $Y_{s}^{\varepsilon }=%
\tilde{Y}_{s/\varepsilon ^{2}}$ is started with $Y_{0}^{\varepsilon }\sim
\nu $; note then that ${\mathbb{E}}\int_{0}^{t}Y_{s}^{\varepsilon }\otimes
Y_{s}^{\varepsilon }=tC$. Furthermore : 
\begin{eqnarray*}
{\mathbb{E}}\left[ \left( \int_{0}^{t}Y_{s}^{\varepsilon
,i}Y_{s}^{\varepsilon ,j} ds\right) ^{2}\right] &=&\int_{[0,t]^{2}}{\mathbb{E%
}}\left[ Y_{s}^{\varepsilon ,i}Y_{s}^{\varepsilon ,j}Y_{r}^{\varepsilon
,i}Y_{r}^{\varepsilon ,j}\right] \,ds\,dr \\
&=&\int_{[0,t]^{2}}{\mathbb{E}}\left[ Y_{s}^{\varepsilon
,i}Y_{s}^{\varepsilon ,j}\right] {\mathbb{E}}\left[ Y_{r}^{\varepsilon
,i}Y_{r}^{\varepsilon ,j}\right] \,ds\,dr \\
&&\;\;\;\;\;\;+2\int_{[0,t]^{2}}{\mathbb{E}}\left[ Y_{r}^{\varepsilon
,i}Y_{s}^{\varepsilon ,j}\right] {\mathbb{E}}\left[ Y_{s}^{\varepsilon
,i}Y_{r}^{\varepsilon ,j}\right] \,ds\,dr \\
&\leq &C_{ij}^{2}t^{2}+\int_{[0,t]^{2}}\left( {\mathbb{E}}\left[
Y_{r}^{\varepsilon ,i}Y_{s}^{\varepsilon ,j}\right] ^{2}+{\mathbb{E}}\left[
Y_{s}^{\varepsilon ,i}Y_{r}^{\varepsilon ,j}\right] ^{2}\right) \,ds\,dr. \\
&\leq &C_{ij}^{2}t^{2}+4\int_{[0,t]^{2}}{\mathbb{E}}\left[
Y_{r}^{\varepsilon }\otimes Y_{s}^{\varepsilon }\right] ^{2}\mathbf{1}%
_{\{s\leq r\}}\,ds\,dr.
\end{eqnarray*}%
Now note that for $s\leq r$, 
\begin{equation*}
{\mathbb{E}}\left[ Y_{r}^{\varepsilon }\big\vert Y_{s}^{\varepsilon }\right]
=e^{-\varepsilon ^{-2}M(r-s)}Y_{s}^{\varepsilon },
\end{equation*}
so that 
\begin{equation*}
{\mathbb{E}}\left[ Y_{r}^{\varepsilon }\otimes Y_{s}^{\varepsilon }\right]
^{2} = O(e^{-\frac{2\lambda }{\varepsilon ^{2}}(r-s)}),
\end{equation*}
and we finally obtain 
\begin{eqnarray*}
{\mathbb{E}}\left[ \left( \int_{0}^{t}Y_{s}^{\varepsilon
,i}Y_{s}^{\varepsilon ,j}ds-tC_{ij}\right) ^{2}\right] &=&{\mathbb{E}}\left[
\left( \int_{0}^{t}Y_{s}^{\varepsilon ,i}Y_{s}^{\varepsilon ,j}ds\right) ^{2}%
\right] -\left( tC_{ij}\right) ^{2} \\
&\leq &\mbox{(const.)}\int_{0}^{t}\left( \int_{s}^{\infty }e^{-\frac{%
2\lambda }{\varepsilon ^{2}}(r-s)}dr\right) ds\;=\;\mbox{(const.)}%
\varepsilon ^{2}.
\end{eqnarray*}

We have thus proved that the $L^{2}$-convergence (and by Gaussian properties 
$L^{q}$-convergence for any $q<\infty $) $\int_{0}^{t}Y_{s}^{\varepsilon
}\otimes Y_{s}^{\varepsilon }ds\rightarrow tC$ is in fact of order $%
\varepsilon $. Actually, from here on it is not difficult to establish
convergence rates of (\ref{eqRPConv}). From \cite[Thm A.13]{FVforth} and the
work already done in the previous proof (reduction from $q^{th}$ moments to
second moments is immediate by Gaussian chaos) it will suffice to check%
\begin{eqnarray}
{\mathbb{E}}\left[ \left\vert W_{s,t}-MX_{s,t}^{\varepsilon }\right\vert ^{2}%
\right] &\leq &(\mbox{const.})\delta \left( \varepsilon \right) |t-s|^{\beta
}, \\
{\mathbb{E}}\left[ \left\vert \widehat{\mathbb{W}}_{s,t}-%
\int_{s}^{t}MX_{s,r}^{\varepsilon }\otimes d(MX^{\varepsilon
})_{r}\right\vert ^{2}\right] &\leq &(\mbox{const.})\delta \left(
\varepsilon \right) \left\vert t-s\right\vert ^{2\beta },
\end{eqnarray}%
for fixed $1/3<\beta \leq 1/2$ and $\delta =\delta \left( \varepsilon ;\beta
\right) $. As long as $\delta \left( \varepsilon \right) \rightarrow 0$ with 
$\varepsilon $, this is also the convergence rate in (\ref{eqRPConv}), for
any $1/3<$ $\alpha <\beta $. A short computation shows that 
\begin{eqnarray*}
{\mathbb{E}}\left[ \left\vert W_{s,t}-MX_{s,t}^{\varepsilon }\right\vert ^{2}%
\right] ^{\frac{1}{2}} &\leq &C\left\vert t-s\right\vert ^{1/2}\text{ and }%
\leq C\varepsilon , \\
{\mathbb{E}}\left[ \left\vert \widehat{\mathbb{W}}_{s,t}-%
\int_{s}^{t}MX_{s,r}^{\varepsilon }\otimes d(MX^{\varepsilon
})_{r}\right\vert ^{2}\right] ^{\frac{1}{2}} &\leq &C\left\vert
t-s\right\vert \text{ and }\leq C\varepsilon .
\end{eqnarray*}%
By (geometric) interpolation, with exponent $2\beta \leq 1$ and $1-2\beta ,$%
\begin{eqnarray*}
{\mathbb{E}}\left[ \left\vert W_{s,t}-MX_{s,t}^{\varepsilon }\right\vert ^{2}%
\right] ^{\frac{1}{2}} &\leq &C\varepsilon ^{1-2\beta }\left\vert
t-s\right\vert ^{\beta }, \\
{\mathbb{E}}\left[ \left\vert \widehat{\mathbb{W}}_{s,t}-%
\int_{s}^{t}MX_{s,r}^{\varepsilon }\otimes d(MX^{\varepsilon
})_{r}\right\vert ^{2}\right] ^{\frac{1}{2}} &\leq &C\varepsilon ^{1-2\beta
}\left\vert t-s\right\vert ^{2\beta },
\end{eqnarray*}%
and we obtain $L^{2}$ (and then $L^{q}$, any $q<\infty $) convergence rate $%
\delta \left( \varepsilon \right) =\varepsilon ^{1-2\beta }$. In other
words, given $\alpha \in \left( 1/3,1/2\right) $ we have rate arbitrarily
close to $1-2\alpha $; by a Borel--Cantelli argument this is also the
almost-sure rate, say, along $\varepsilon =1/n$. Note that by working in the
stronger topology ($\alpha \uparrow 1/2$), one loses on the rate of
convergence. Also, since "level-2" rough path theory imposes $\alpha >1/3$,
an upper bound for the best possible rate (in a rough path metric!) is given
by $1-2/3=1/3$. (The situation is very similar to the rate of convergence of
piecewise linear approximations to Brownian motion\ (as rough path), see 
\cite{HN08}.
\end{remark}

\section{Applications and remarks}

\bigskip We conclude this note with a number of applications and remarks.

\begin{remark}
Write the (anti-symmetric) matrix $MC-CM^{\ast }$ as $\sum_{i<j}\gamma
_{i,j}[e_{i},e_{j}]$. As a consequence of our main theorem, we have the
following convergence result for ODEs driven by "physical" Brownian motion,
in the zero mass limit.Let $V_{0}$ $\in $ $\func{Lip}^{1}$, $V=(V_{1},\ldots
,V_{n})$ $\in $ $\func{Lip}^{2+\delta }$ for some $\delta >0$, be vector
fields on ${\mathbb{R}}^{e}$. Let $Y^{\varepsilon }$ be the solution to the
SDE (actually, random ODE) 
\begin{equation*}
dY^{\varepsilon }=V_{0}(Y^{\varepsilon
})dt+\sum_{i=1}^{d}V_{i}(Y^{\varepsilon })dX^{\varepsilon
;i},\,\,Y_{0}^{\varepsilon }=x\in {\mathbb{R}}^{e}.
\end{equation*}%
Then as $\varepsilon \rightarrow 0$, $Y^{\varepsilon }$ converges to the
solution $Y$ of the following It\^{o} SDE 
\begin{equation*}
dY=\tilde{V}_{0}\left( Y\right) dt+\sum_{i=1}^{n}\tilde{V}%
_{i}(S^{0})dB_{s}^{i},
\end{equation*}%
where $(\tilde{V}_{1},\ldots ,\tilde{V}_{n})=(V_{1},\ldots ,V_{n})M^{-1}$
and 
\begin{equation*}
\tilde{V}_{0}=V_{0}+\frac{1}{2}\sum_{i=1}^{d}\tilde{V}_{i}\tilde{V}_{i}+%
\frac{1}{2}\sum_{1\leq i<j\leq n}\gamma _{i,j}[\tilde{V}_{i},\tilde{V}_{j}].
\end{equation*}%
Indeed, given our main result, this is a simple consequence of rough path
stability and the impact of higher order perturbations on RDEs, see e.g. 
\cite[Ch. 12]{FVforth}, combined with the usual It\^{o}--Stratonovich
correction.
\end{remark}

\begin{remark}[Magnetohydrodynamics]
A (physical) system - as described by the above differential equations -
which is driven/controlled by a single, charged Brownian particle in a
magnetic field may not appear to be a standard situation in applied science.
However, it is not hard to think of a system being influenced by a cloud of
such particles. If $N$ such particles move independently, our main theorem
applies immediately in dimension $n=3N$. Of course, for $N>>1$ one needs to
incorporate interactions between the particles. In fact, the movement of a
cloud of charged particles will effect very much the magnetic field itself.
In our model, the matrix $M$ should then be allowed to depend on the (bulk)
behaviour of the $N$ particles. Much more work will be necessary to give a
proper rough path analysis of this situation, we do remark however that
first mean field results for rough differential equations have been obtained
by Cass--Lyons \cite{CL}. It is quite conceivable that our rough path
perspective then provides a very novel point of view for
magnetohydrodynamics.
\end{remark}

\begin{remark}
Similar correction terms appear when one considers limits of one-forms
integrated against $X^{\varepsilon }$. The rough path correction matters,
for instance, if one wants to make a change of measure and represent the
stochastic integral in the Girsanov factor in terms of $X^{\varepsilon }$,
rather than in terms of the underlying Brownians. (This type of
representation plays an important role in "robust" filtering, path sampling
of conditioned diffusions \cite{HSV, H} and related topics.)
\end{remark}

\begin{remark}
The \emph{signature} of a path is the sequence of its iterated integrals
against itself. For (deterministic) paths of bounded variation, it fully
characterizes the path up to "tree-like" equivalence \cite{HL}. In a similar
spirit, the \emph{expected} signature of a processes (in the sense below)
characterizes the essence of its law, at least when it comes the solution of
stochastic differential equations; the so-called cubature method is based on
this \cite{LVcub}. In \cite{NH} expected signatures of many stochastic
processes are considered. By either specializing these considerations or a
direct computation one can see that, as $\varepsilon \rightarrow 0$, the
expected signature of $MX^{\varepsilon }$ converges to 
\begin{equation*}
{\mathbb{E}}\left[ S(\widehat{\mathbf{W}})_{0,T}\right] =\exp \left( \frac{T%
}{2}\left( \sum_{i}e_{i}\otimes e_{i}+\sum_{i<j}\gamma
_{i,j}[e_{i},e_{j}]\right) \right) .
\end{equation*}
\end{remark}

\begin{remark}
In this paper we have only considered the case where $M$ was constant. It is
however natural to consider the case where the dynamics (i.e. friction,
magnetic field and covariance matrix of the brownian term) depend on the
position $X$ of the particle. This leads to consider the coupled system of
SDEs 
\begin{eqnarray*}
dY^\varepsilon &=& - \varepsilon^{-2} M(X^\varepsilon) Y^\varepsilon dt +
\varepsilon^{-1} dW_t, \\
dX^\varepsilon &=& \varepsilon^{-1} Q(X^\varepsilon) Y^\varepsilon dt.
\end{eqnarray*}

It is then not difficult to show that the paths $X^\varepsilon$ $\to$ $X^0$
pointwise, where $X^0$ is solution to 
\begin{eqnarray*}
dX^0 &=& \left[ D_x (QM^{-1}) : C \right](X^0) dt + (QM^{-1})(X^0) dW,
\end{eqnarray*}
where $\left[ D_x (QM^{-1}) : C \right]^i$ $:=$ $\sum_{k,l} \partial_k
(QM^{-1})^i_l C^{kl}$ and $C=C(x)$ is still defined by the same formula. 

As for the iterated integral, we obtain a similar correction except that it
is now state dependent (and thus random): 
\begin{equation*}
\int_{0}^{t}X^{\varepsilon }\otimes dX^{\varepsilon }\rightarrow
\int_{0}^{t}X^{0}\otimes \circ dX^{0}+\frac{1}{2}\int_{0}^{t}Q\left(
C(M^{-1})^{\ast }-M^{-1}C\right) Q^{\ast }(X_{s}^{0})ds.
\end{equation*}%
The computation of these pointwise convergences is close to the beginning of
the proof of Theorem 1, but instead of using the ergodic theorem one should
notice that for small $\varepsilon $ 
\begin{equation*}
\int_{0}^{t}f(X_{s}^{\varepsilon },Y_{s}^{\varepsilon })ds\approx
\int_{0}^{t}(\int f(X_{s}^{\varepsilon },y)\nu _{C}(X_{s}^{\varepsilon
},dy))ds,
\end{equation*}%
as the process $Y^{\varepsilon }$ evolves at a much faster time-scale than $%
X^{\varepsilon }$ (here $\approx $ should mean that the difference is small
in $L^{2}$-norm). 
%
The detailed verification of convergence in rough path sense is technical
(mainly, because one looses the Gaussian and Markovian structure of $%
Y^{\varepsilon }$) and left to subsequent work.
\end{remark}

\begin{remark}[On the role of Brownian roughness] \label{rmkFBM}
Let us return to (\ref{newton}), but now with Brownian motion replaced by a
(deterministic) path $\gamma $ defined on $\left[ 0,T\right] $. That is, we
consider the evolution%
\begin{equation*}
m\ddot{x}=-M\dot{x}+\dot{\gamma}.
\end{equation*}
Of course, even when $\gamma $ fails to be differentiable this equation is
well-defined in the distributional sense, thanks to the additive structure
of the noise. As before we assume that $M$ has an antisymmetric part which
therefore triggers rotation and thus effects the area. One may wonder how
"rough" the driving force (now assumed deterministic) needs to be to see
some non-trivial area correction in the limit. As we now show, the roughness
of Brownian motion -with (almost) $1/2$-H\"{o}lder sample paths - is crucial
and \textit{no area corrections} arises when $\gamma $ is H\"{o}lder with
exponent greater than $1/2$.
\end{remark}

\begin{proposition}
Assume $\gamma $ is $\alpha $-H\"{o}lder. Then, as $m$ $\rightarrow $ $0$, $%
Mx_{0,\cdot }$ converges to $\gamma _{0,\cdot }$ in $\beta $-H\"{o}lder
topology for any $\beta $ $<$ $\alpha $. In particular, when $\alpha >1/2$
it follows that%
\begin{equation*}
\int_{0}^{\cdot }Mx_{0,t}\otimes d\left( Mx_{t}\right) \rightarrow
\int_{0}^{\cdot }\gamma _{0,\cdot }\otimes d\gamma \text{ as }m\rightarrow 0
\end{equation*}%
where the integral on the left-hand-side is understood in Young sense,
convergence may be understood uniformly on $\left[ 0,T\right] $ (and
actually in $\beta $-H\"{o}lder rough path sense).
\end{proposition}

\begin{proof}
By interpolation, it is enough to establish pointwise convergence of $%
Mx_{0,\cdot }$ to $\gamma _{0,\cdot }$ in conjunction with uniform $\alpha $%
-H\"{o}lder bounds. Equivalently, we want to show that, pointwise and with
uniform $\alpha $-H\"{o}lder bounds,%
\begin{equation*}
z:=\gamma _{0,\cdot }-Mx_{0,\cdot }\rightarrow 0\text{ as }m\rightarrow 0.
\end{equation*}%
Note that%
\begin{equation*}
dz_{t}=-\frac{M}{m}z_{t}dt+d\gamma _{t}
\end{equation*}%
from which we see, writing $z_{s,t}=z_{t}-z_{s}$ as usual, 
\begin{equation*}
z_{s,t}=(e^{-\frac{M}{m}(t-s)}-I)z_{s}+\int_{s}^{t}e^{-\frac{M}{m}%
(t-r)}d\gamma _{r}.
\end{equation*}%
The last integral is a Young (actually Riemann-Stieltjes) integral, for its
integrand has finite variation. To see this, note that for $s\leq t$, $%
|e^{-tM}-e^{-sM}|\leq Ce^{-\lambda s}\left( (t-s)\wedge 1\right) $ where $%
\lambda $ is, as in previous sections, a lower bound on the real part of the
spectrum of $M$. We may then estimate, for any subdivision $0\leq s_{0}\leq
\ldots \leq s_{N}<\infty $ 
\begin{eqnarray*}
\sum_{i}|e^{-s_{i+1}M}-e^{-s_{i}M}| &\leq &C\sum_{i}e^{-\lambda s_{i}}\left(
(s_{i+1}-s_{i})\wedge 1\right) \\
&\leq &C\sum_{n=0}^{\infty }e^{-\lambda n}\sum_{i}\mathbf{1}_{\{n\leq
s_{i}<n+1\}}\left( (s_{i+1}-s_{i})\wedge 1\right) \\
&\leq &2C\sum_{n=0}^{\infty }e^{-\lambda n}\;\;<\;\;\infty .
\end{eqnarray*}%
In particular, it follows that%
\begin{equation*}
\sup_{0<m\leq 1}\left\Vert e^{-\frac{M}{m}\cdot }\right\Vert _{1\text{-var}%
;[0,T]}<\infty .
\end{equation*}%
We now address pointwise convergence. Since $z_{0}=0$ we can estimate,
whenever $0<\delta \leq t\leq T$, 
\begin{eqnarray*}
|z_{t}| & = & \left\vert \int_{0}^{t}e^{-\frac{M}{m}(t-r)}d\gamma
_{r}\right\vert%
\leq e^{-\frac{\lambda }{m}\delta }\left\vert \int_{0}^{t-\delta }e^{-%
\frac{M}{m}((t-\delta )-r)}d\gamma _{r}\right\vert +\left\vert
\int_{t-\delta }^{t}e^{-\frac{M}{m}(t-r)}d\gamma _{r}\right\vert \\
&\leq &Ce^{-\frac{\lambda }{m}\delta }\Vert \gamma \Vert _{\alpha \text{-H%
\"{o}l;}\left[ 0,T\right] }T^{\alpha }\left( 1+\left\Vert e^{-\frac{M}{m}%
\cdot }\right\Vert _{1-var;[0,T]}\right) \\
&&+C\Vert \gamma \Vert _{\alpha \text{-H\"{o}l;}\left[ 0,T\right] }\delta
^{\alpha }\left( 1+\left\Vert e^{-\frac{M}{m}\cdot }\right\Vert _{1\text{-var%
};[0,T]}\right) \\
&\leq &C\left( e^{-\frac{\lambda }{m}\delta }+\delta ^{\alpha }\right) ,
\end{eqnarray*}%
where $C$ is a constant which does not depend on $m$. Taking%
\begin{equation*}
\delta =\frac{\alpha m}{\lambda \log (1/m)}
\end{equation*}%
one sees that 
\begin{equation*}
|z_{t}|\leq Cm^{\alpha }\left( 1+|\log m|\right) ^{\alpha }\;\;\;\leq
\;\;\;Cm^{\beta },
\end{equation*}%
which in particular gives us pointwise convergence. As for uniform H\"{o}%
lder bounds, take $s\leq t$ so that 
\begin{equation*}
z_{s,t}=\int_{s}^{t}e^{-\frac{M}{m}(t-r)}d\gamma _{r}+(e^{-\frac{M}{m}%
(t-s)}-I)z_{s}.
\end{equation*}%
As before, the integral term is bounded by $C(t-s)^{\alpha }$. For the other
term, note that 
\begin{eqnarray*}
\left\vert (e^{-\frac{M}{m}(t-s)}-I)z_{s}\right\vert &\leq &C\left( \frac{t-s%
}{m}\wedge 1\right) |z_{s}| \\
&\leq &C\left( \frac{t-s}{m}\right) ^{\beta }|z_{s}|\;\;\;\leq
\;\;\,C(t-s)^{\beta },
\end{eqnarray*}%
where we have used the previous point wise estimate on $z$ in the last
inequality. This proves that the paths $z$ are uniformly $\beta $-H\"{o}lder
continuous and finishes the proof.
\end{proof}

The above proposition shows, for instance, that replacing Brownian motion in
our main theorem by \textit{fractional} Brownian motion with Hurst parameter 
$H>1/2$ will not allow for a similar statement with non-trivial stochastic
area correction. (It is recalled that fractional Brownian motion has H\"{o}%
lder continuous sample path with exponent arbitrarily close to $H$.) Let us,
finally and briefly, discuss a similar statement when Brownian motion in our
main theorem is replaced by a finite energy path; that is, a path%
\begin{equation*}
\gamma :\left[ 0,T\right] \rightarrow \mathbb{R}^{n}.
\end{equation*}%
which may be written as indefinite integral of some function in $L^{2}\left( %
\left[ 0,T\right] ,\mathbb{R}^{n}\right) $, which we shall call $\dot{\gamma}
$. By Cauchy-Schwarz, such finite energy paths are guaranteed to be $1/2$-H%
\"{o}lder but, in general, one does not have better H\"{o}lder regularity.
In particular, since the area is not continuous in $1/2$-H\"{o}lder
topology, the above proposition just about fails to cover finite energy
paths. A direct argument, however, is not difficult. As in the above proof,
we set 
\begin{equation*}
z:=\gamma _{0,\cdot }-Mx_{0,\cdot }
\end{equation*}%
and note from the previous argument $|z_{t}|\leq Cm^{1/2}\left( 1+|\log
m|\right) ^{1/2}$, uniformly over $t\in \left[ 0,T\right] $. We then write$%
\; $%
\begin{equation*}
\dot{z}_{t}=-\frac{M}{m}z_{t}+\dot{\gamma}_{t}
\end{equation*}%
and take the scalar product in $\mathbb{R}^{n}$ with $\dot{z}_{t}$,
following by integration over $\left[ 0,T\right] $ to see that%
\begin{eqnarray*}
\int_{0}^{T}\left\vert \dot{z}_{t}\right\vert ^{2}dt &=&-\frac{M}{m}\frac{1}{%
2}\left\vert z_{T}\right\vert ^{2}+\int_{0}^{T}\left\langle \dot{\gamma}_{t},%
\dot{z}_{t}\right\rangle dt \\
&\leq &\frac{1}{2}\int_{0}^{T}\left\vert \dot{\gamma}_{t}\right\vert ^{2}dt+%
\frac{1}{2}\int_{0}^{T}\left\vert \dot{z}_{t}\right\vert ^{2}dt.
\end{eqnarray*}%
This implies a uniform (in $m$) $L^{2}$-bound on $\dot{z}$. This of course
implies a uniform $L^{1}$-bound on $\dot{z}$ and thus a uniform $1$%
-variation bound on $z$. Knowning that $z$ converges to zero uniformly on $%
\left[ 0,T\right] $, it now follows from interpolation that this
convergences also takes place in $p$-variation, for any $p>1$. Now, the area
is a continuous function of the underlying paths in $p$-variation as along
as $p<2$ and so we can conclude: replacing Brownian motion in our main
theorem by a finite energy (also known as Cameron--Martin) path will not
allow for a similar statement with non-trivial stochastic area correction.

\section{Appendix: elements of rough path theory}

A rough path on an interval $[0,T]$ with values in a Banach space $V$ consists of a continuous function $X\colon \lbrack 0,T]\rightarrow V$, as
well as a continuous \textquotedblleft second order
process\textquotedblright\ $\mathbb{X}\colon \lbrack 0,T]^{2}\rightarrow
V\otimes V$, subject to certain (i) algebraic and (ii) analytic conditions.
Towards (i), the behaviour of iterated integrals of smooth paths suggests to
impose the algebraic relation ("\textit{Chen's relation}"), 
$$ \XX_{s,t} - \XX_{u,t} -\XX_{s,u} =  X_{s,u}\otimes X_{u,t}\;,
$$
assumed to hold for every triple of times $(s,u,t)$. Since $%
X_{t,t}=0$, it immediately follows (take $s=u=t$) that we also have $\mathbb{%
X}_{t,t}=0$ for every $t$. One should think of $\mathbb{X}$ as \textit{%
postulating} the value of the quantity 
$$ \int_s^t X_{s,r}\otimes dX_r = XX_{s,t} \;,
$$ where we take the right hand side as a \textit{definition} for the
left hand side. We insist that knowledge of the path
$t\mapsto (X_{0,t},\mathbb{X}_{0,t})$ already determines the entire second order 
process $\mathbb{X}$. In this sense $(X,\mathbb{X})$ is indeed a path, and not some 
two-parameter object.

Note that the algebraic relations are by themselves not
sufficient to determine $\mathbb{X}$ as a function of $X$. Indeed, for any $%
V\otimes V$-valued function $F$, the substitution $\mathbb{X}_{s,t}\mapsto 
\mathbb{X}_{s,t}+F_{t}-F_{s}$ leaves the left hand side of the above algebraic 
relation invariant. We will see later on how one should interpret such a
substitution. The aim of imposing these algebraic relations is to ensure that %
$\XX$ does indeed have the basic addivity properties of any
(reasonable) integral when considering it over two adjacent intervals.

For $\alpha \in ({\frac{1}{3}},{\frac{1}{2}}]$, one can denote by $\DD%
^{\alpha }([0,T],V)$ the space of those rough paths $(X,\mathbb{X})$ such
that  $$ \|X\|_{\alpha} = \sup_{s \neq t \in [0,T]} {|X_{s,t}|\over |t-s|^{\alpha}} < \infty\;,\qquad
\|\XX\|_{2\alpha} =  \sup_{s \neq t \in [0,T]} {|\XX_{s,t}|\over |t-s|^{2\alpha}} < \infty\;.
$$ If one ignores the non-linear, algebraic constraint there is
a natural way to think of $(X,\mathbb{X})$ as element in the Banach space of
such maps with (semi-)norm $\Vert X\Vert _{\alpha }+\Vert \mathbb{X}\Vert
_{2\alpha }$. However, due to the non-linear algebraic relation $\DD%
^{\alpha }$ is not a linear space, although a closed subset of the
aforementioned Banach space. 
\begin{definition}
Given rough paths $\bX,\bY\in \DD^{\alpha }([0,T],V)$, we define the
(inhomogenous) $\alpha $-H\"{o}lder rough path metric 
\begin{equation*}
\rho _{\alpha }(\bX,\bY):=\sup_{s\neq t\in \lbrack 0,T]}{\frac{%
|X_{s,t}-Y_{s,t}|}{|t-s|^{\alpha }}}+\sup_{s\neq t\in \lbrack 0,T]}{\frac{|%
\mathbb{X}_{s,t}-\YY_{s,t}|}{|t-s|^{2\alpha }}}.
\end{equation*}%
{}
\end{definition}
Let us note that $\DD^{\alpha }([0,T],V)$ so becomes a complete, metric space. 
The perhaps cheapest way to show convergence with respect to this rough path
metric is based on interpolation: in essence, it is enough to establish
pointwise convergence in conjunction with uniform "rough path"  $\alpha$-H\"older 
bounds. 
We conclude this part with two important remarks. First, we can ask
ourselves up to which point the algebraic relations are already
sufficient to determine $\mathbb{X}$. Assume that we can associate to a
given function $X$ two different second order processes $\mathbb{X}$ and $%
\bar{\XX}$, and set $G_{s,t}=\mathbb{X}_{s,t}-\bar{\XX}_{s,t}$. It then
follows immediately that 
$$ G_{s,t} = G_{u,t} + G_{s,u}\;$$
so that in particular $G_{s,t}=G_{0,t}-G_{0,s}$. We conclude that
 $\mathbb{X}$ is in general determined
only up to the increments of some function $F$ with values in $(V\otimes V)$
and H\"older continuous with exponent $2\alpha$. The choice of $F$ \textit{does} usually matter and there is in general no
obvious canonical choice.

The second remark is that this construction can possibly be useful only if $%
\alpha \leq {\frac{1}{2}}$. Indeed, if $\alpha >{\frac{1}{2}}$, then a
canonical choice of $\mathbb{X}$ is given in terms of the Young integral.
Furthermore, it is clear in this case that $\mathbb{X}$ must
be unique, since any additional increment should be $2\alpha $-H\"{o}lder
continuous, which is of course only possible if $\alpha
\leq {\frac{1}{2}}$. This is however \textit{not} to say that $\mathbb{X}$
is uniquely determined by $X$ if the latter is smooth, when interpreted as
an element of $\DD^{\alpha }$. Indeed, if $\alpha \leq {\frac{1}{2}}$, $F$
is any $2\alpha $-H\"{o}lder continuous function with values in $V\otimes V$
and $\mathbb{X}_{s,t}=F_{t}-F_{s}$, then the path $(0,\mathbb{X})$ is a
perfectly \textquotedblleft legal\textquotedblright\ element of $\DD^{\alpha
}$, even though one cannot get any smoother than the function $0$. The
impact of perturbing $\mathbb{X}$ by some $F\in \mathcal{C}^{2\alpha }$ in
the context of differential equations and integration is dramatic: additional
drift terms in (Lie-bracket) directions can appear; the famous It\^{o}-Stratonovich 
correction is also understood from this picture. The reader may find (much) more
in \cite{LQ02, FVforth} and \cite{FH}.

\end{document}